\documentclass[11pt]{article}
\usepackage[dvips]{graphicx}

\usepackage{amsmath}
\usepackage{multicol}
\usepackage{amsfonts}
\usepackage{amsthm}
\usepackage{amssymb}
\usepackage{lettrine}
\usepackage[sort]{cite}
\makeatletter
\renewenvironment{proof}[1][\proofname]{\par
  \pushQED{\qed}%
  \normalfont 
  \trivlist
  \item[\hskip 1.6em
        \bfseries
    #1]\hskip .5em\ignorespaces
}{%
  \endtrivlist\@endpefalse
} \makeatother

\usepackage{times}
\usepackage{booktabs}
\usepackage{caption}
\newtheorem{thm}{Theorem}[section]

\newtheorem{prop}[thm]{Proposition}
\newtheorem{cor}[thm]{Corollary}
\newtheorem{rmk}[thm]{Remark}

\date{}
\begin{document}

\title{A stochastic maximum principle for backward delayed system via advanced stochastic differential equation (ASDE)
\thanks{The first author acknowledges the support from the Fundamental Research Funds for the Central
Universities (2010QS05), P. R. China. The first author also thanks Department of Applied Mathematics, The Hong Kong Polytechnic University for their hospitality during
her visit to Hong Kong. The second author acknowledges the support of RGC Earmarked grant 501010 and research fund of Hong Kong Polytechnic University (A-PL14).}
\author{Li Chen\footnote{Department of Mathematics, China University of Mining $\&$ Technology, Beijing 100083, China.
 Email address: chenli@cumtb.edu.cn.}\,\,\,\,\,
        Jianhui Huang\footnote{ Hong Kong Polytechnic University. Email address: majhuang@inet.polyu.edu.hk.}}}
\maketitle

\begin{abstract}The main contributions of this paper are three fold. First, our primary concern is to investigate a class of stochastic recursive delayed control problems which arise naturally with sound backgrounds but have not been well-studied yet. For illustration purpose, some concrete examples are also provided here. We derive the stochastic maximum principle of sufficient condition to the optimal control in both cases with and without control delay. Second, it is interesting that a new class of time-advanced stochastic differential equations (ASDEs) is introduced as the adjoint process via duality relation. To our best knowledge, such equations have never been discussed in literature although they possess more academic values besides the control study here. Some existence and uniqueness result to ASDEs is presented. Third, to illustrate our theoretical results, some dynamic optimization problems are discussed based on our stochastic maximum principles. In particular, the optimal controls are derived explicitly by solving the associated time-advanced ordinary differential equation (AODE), the counterpart of the ASDE in its deterministic setup.
\end{abstract}


\section{Introduction}Our starting point is the following backward stochastic differential equation (BSDE) with time-delayed generator:\begin{equation}\label{MBD01}\left\{\begin{aligned}
-dy(t)=&\ f(t, y(t), y(t-\delta), z(t), z(t-\delta))dt-z(t)dW(t),\,\quad t\in[0,T],\\
y(T)=&\ \xi,\\
y(t)=&\ \varphi(t),\,z(t)=\psi(t),\,\quad t\in[-\delta,0).\end{aligned}\right.
\end{equation}Two remarkable features of Eq.(\ref{MBD01}): (i) The terminal instead initial condition is specified; (ii) The generator $f$ depends not only on the instantaneous state $(y(t), z(t))$ but also on $(y(t-\delta), z(t-\delta))$ through the time-delayed parameter $\delta>0.$ The feature (i) makes Eq.(\ref{MBD01}) essentially different to the well-studied stochastic delay differential equation (SDDE) (see e.g. Mohammed \cite{m1}, \cite{m2}, etc.) in which the initial state condition is given beforehand. Eq.(\ref{MBD01}) also differs from the standard BSDE due to its time-delayed generator from (ii). In particular, it distinguishes from the anticipated backward stochastic differential equation (ABSDE) introduced by Peng and Yang \cite{py} which is the duality of SDDE. Eq.(\ref{MBD01}) is first introduced by Delong and Imkeller \cite{di} and it has many real backgrounds in economics, finance, management, or other decision sciences. More details can be found in Delong \cite{d1}, \cite{d2}, Delong and Imkeller \cite{di} and the reference therein. Due to the interesting structure and wide-range applications, it is very natural and necessary to study the dynamic optimizations of Eq.(\ref{MBD01}). However, to our best knowledge, very few works have been done along this direction thus we aim to fill this research gap in some systematic way. To this end, we study the following more general controlled backward delayed system:\begin{equation}\label{MBD02}\left\{\begin{aligned}
-dy(t)=&\ f\left(t,y(t),\int_{t-\delta}^{t}\phi(t,s)y(s)\alpha(ds),z(t), \int_{t-\delta}^{t}\phi(t,s)z(s)\alpha(ds), v(t), \int_{t-\delta}^{t}\phi(t,s)v(s)\alpha(ds)\right)dt\\&-z(t)dW(t),\,\quad t\in[0,T],\\
y(T)=&\ \xi,\\
y(t)=&\ \varphi(t),\,z(t)=\psi(t),\,\quad t\in[-\delta,0).\end{aligned}\right.
\end{equation}Here, $\delta$ is time delay parameter, $\alpha$ is some $\sigma$-finite measure and $\phi(\cdot,\cdot)$ is some bounded process. The relevance and importance of our optimization problems can be illustrated by the following concrete examples.\vspace{0.2cm}

{\it Example 1.1} (\emph{Optimization of recursive utility with moving average}) This example originates from Delong \cite{d1} in which the decision makers have recursive utility with moving average generators. Such utility can be used to characterize the habit information, disappointment effects as well as volatility aversion in decision-making. Accordingly, the objective of decision maker is to maximum his/her utility by selecting suitable instantaneous consumption process $c(t)$. This leads to the following dynamic optimization problem$$\inf_{c(\cdot) \in \mathcal{U}_{ad}}y^{c}(0)$$where the recursive utility $y(t)$ satisfies the following BSDE with time-delayed generator\begin{equation}\label{ex1}\left\{\begin{aligned}
-dy(t)=&\ f\left(t, y(t), \frac{1}{t}\int_{0}^{t}y(s)ds, z(t), \frac{1}{t}\int_{0}^{t}z(s)ds, c(t)\right)dt\\&-z(t)dW(t),\,\quad t\in[0,T],\\
y(T)=&\ \xi,\\
y(t)=&\ \varphi(t),\,z(t)=\psi(t),\,\quad t\in[-\delta,0).\end{aligned}\right.
\end{equation}Eq.(\ref{ex1}) can be viewed as the special case of Eq.(\ref{MBD02}) by noting $\frac{1}{t}\int_{0}^{t}y(s)ds=\int^{t}_{t-T}y(s)\frac{T}{t}\chi_{\{s\geq0\}}\alpha(ds)$ where $\alpha$ is uniform measure on $[t-T,t]$. It can characterize the non-monotonic utility to volatility aversion.\vspace{0.2cm}

{\it Example 1.2} (\emph{Pension fund with delayed surplus}) This example comes from Federico \cite{f} where the pension fund manager can invest two assets: the riskless asset $P_0(t)$ satisfies $dP_0(t)=rP_0(t)dt$ with instantaneous return rate $r \geq 0,$ and the risky asset $P_1(t)$ satisfies\begin{equation*}dP_1(t)=\mu P_1(t)dt+\sigma P_1(t)dW(t),\end{equation*}with return rate $\mu \geq r,$ volatility rate $\sigma>0.$ Here, $W(\cdot)$ is a standard Brownian motion. Denote $\lambda=\frac{\mu-r}{\sigma}$ the risk premium, $\theta(t) \in[0,1] $ the proportion of fund invested in risky asset, and $S(t)$ the surplus premium to fund members. Suppose the wealth of pension fund at time $t$ is $y(t)$, and it is reasonable to assume $S(t)$ depends on the performance of fund growth during the past period. Thus, we assume:$$S(t)=g\left(y(t)-\kappa y(t-\delta)\right)$$for some $\kappa>0$ and $g:\mathbb{R}\rightarrow[0,+\infty)$ which is increasing, convex and Lipschitz continuous, $\delta>0$ is the time delay. On the other hand, there should be some running cost or consumption for fund management, which is represented by the instantaneous rate $c(t)$. Hence the wealth process $y(t)$ evolves as:\begin{equation}\label{MBD1}
\left\{\begin{aligned}
dy(t)=&\ \left([\theta(t)\sigma\lambda+r]y(t)-g(y(t)-\kappa y(t-\delta))-c(t)\right)dt+\sigma\theta(t)y(t)dW(t),\, \quad t\in[0,T],\\
y(0)=&\ y_0,\\
y(t)=&\ 0,\, \quad t\in[-\delta,0).\\
\end{aligned}\right.\end{equation}Note that in practice, the pension fund will be required to provide some minimum guarantee, i.e., to pay some part of the due benefits $\xi$ (which is some random variable) at some given future time $T$. Keep this in mind, the objective of fund manager is to choose $\theta(t)$ and $c(t)$ to reach terminal condition $y(T)=\xi$, and also maximize some given cost functional at the same time. By setting $z(t)=\sigma\theta(t)y(t)$, Eq.({\ref{MBD1}) can be reformulated by the following controlled backward delayed system:\begin{equation}\label{MBD2}
\left\{\begin{aligned}
dy(t)=&\ \{ry(t)+\lambda z(t)-g(y(t)-\kappa y(t-\delta))-c(t)\}dt+z(t)dW(t),\, \quad t\in[0,T],\\
y(t)=&\ 0,\, \quad t\in[-\delta,0),\\
y(T)=&\ \xi.
\end{aligned}\right.
\end{equation}Eq.(\ref{MBD2}) is a special case of (\ref{MBD02}) by setting $\alpha(ds)$ to be Dirac measure at $-\delta,$ the pointwise delay with lag $\delta.$\vspace{0.2cm}

{\it Example 1.3} It is remarkable that there exist considerably rich literature to discuss the controlled stochastic delay differential equations (SDDEs) (see e.g. \cite{cw}, \cite{l}, \cite{os} etc.) which arise naturally due to the time lag between the observation and regulator, or the possible aftereffect of control. The SDDEs and its optimization have attracted extensive research attention in last few decades, and have been applied in wide-range domains including physics, biology and engineering, etc. (see \cite{m1}, \cite{m2} for more details). Note that these works are discussed in the forward setup because the initial condition is given as the priori. On the other hand, as suggested by Kohlmann and Zhou \cite{kz}, Ma and Yong \cite{my}, the forward controlled systems can be reformulated into some backward controlled systems under mild conditions. For example, in case of some state constraints (e.g. no short selling), it is better to reformulate the controlled forward systems into some backward systems which are more convenient to be analyzed in some cases (see Ji and Zhou \cite{jz}, El. Karoui, Peng and Quenez \cite{kpq}). Also, inspired by Lim and Zhou \cite{lz}, we aim to investigate the following controlled linear backward delayed system:\begin{equation*}\left\{\begin{aligned}dy(t)=&\ \left(\beta_1 y(t)+\beta_2 y(t-\delta)+\gamma_1 z(t)+\gamma_2 z(t-\delta)+\alpha v(t)\right)dt+z(t)dW(t),\, \quad \quad t\in[0,T],\\
y(T)=&\ \xi,\, \quad \quad t\in[-\delta,0],\\\end{aligned}\right.\end{equation*}which can be viewed as the linear constrained forward controlled delay system by using penalty approach, or the limit of a family of linear unconstrained forward delayed system. \vspace{0.2cm}

The rest of this paper is organized as follows. In Section $2$, we introduce the advanced stochastic differential equation (ASDE). Some preliminary results on ASDE and the associated BSDE with delay generator are also given. The stochastic recursive delayed control problems are formulated in Section $3$, and two maximum principles are derived based on the duality between the ASDE and the BSDE with delayed generator. As the application of our theoretical results, in Section $4-6$ we revisit some motivating examples given in Section $1$ and the optimal controls are derived explicitly by solving the associated time-advanced ordinary differential equation (AODE).

\section{Notations and Preliminaries}
Let $T>0$ be some finite time horizon. For any Euclidean space H, we denote by $\langle\cdot,\cdot\rangle$ (resp. $|\cdot|$) the
scalar product (resp. norm) of H. Let $\mathbb{R}^{n\times m}$ be the Hilbert space
of all $n\times m$ matrices with the inner product$$\langle A,B\rangle:=tr\{AB^{\top}\},\,\,\,\,\forall A,B\in\mathbb{R}^{n\times m}.$$Here the superscript $\top$ denotes the transpose of vector or matrix. Let $W(\cdot)$ be a standard $d$-dimensional Brownian motion on a complete probability space $(\Omega, \mathcal{F}, \mathbb{P})$. The information structure is given by the filtration $\mathbb{F}=\{\mathcal{F}_t\}_{t\geq0}$ which is generated by $W(\cdot)$ and augmented by all $\mathbb{P}$-null sets. For $p\geq1$, the following notations are used
throughout this paper:$$\begin{array}{lll}
L^p(\Omega, \mathcal{F}_t, \mathbb{P};H)&:=&\{\xi\,\textrm{is H-valued}\,\mathcal{F}_t-\textrm{measurable random
variable satisfying}\,\,\mathbb{E}[\xi^p]<+\infty\};\\
L^p_{\mathbb{F}}(t_1,t_2;H)&:=&\{\varphi(t),\,t_1\leq t\leq
t_2,\,\textrm{is }\,\mathbb{F}-\textrm{adapted process satisfying}\, \,\mathbb{E}\int^{t_2}_{t_1}|\varphi(t)|^pdt<+\infty\},\\
L^{\infty}_{\mathbb{F}}(t_1,t_2;H)&:=&\{\varphi(t),\,t_1\leq t\leq
t_2,\, \textrm{is H-valued}\,
\mathbb{F}-\textrm{adapted bounded process}\}.\\
\end{array}$$We set $$y_{\delta}(t)=\int_{t-\delta}^{t}\phi(t,s)y(s)\alpha(ds),\quad \quad
z_{\delta}(t)=\int_{t-\delta}^{t}\phi(t,s)z(s)\alpha(ds).$$
Then the backward delayed system of form (\ref{MBD02}) can be rewritten as
\begin{equation}\label{MBD03}\left\{\begin{aligned}
-dy(t)=&\ f\left(t,y(t),y_{\delta}(t),z(t), z_{\delta}(t)\right)dt-z(t)dW(t),\,\quad t\in[0,T],\\
y(T)=&\ \xi,\\
y(t)=&\ \varphi(t),\,z(t)=\psi(t),\,\quad t\in[-\delta,0).\end{aligned}\right.
\end{equation}We introduce the following assumptions:

\textbf{(H2.1)} The function $f:\Omega\times[0,T]\times\mathbb{R}^n\times\mathbb{R}^n\times\mathbb{R}^{n\times d}\times\mathbb{R}^{n\times d}\rightarrow
\mathbb{R}^n$ is $\mathbb{F}$-adapted and satisfies$$\begin{aligned}
&|f(t,y,y_{\delta},z,z_{\delta})-f(t,y',y'_{\delta},z',z'_{\delta})| \leq \ C(|y-y'|+|y_{\delta}-y'_{\delta}|+|z-z'|+|z_{\delta}-z'_{\delta}|)
\end{aligned}$$for any $y,y_{\delta},y',y'_{\delta}\in\mathbb{R}^n$, $z,z_{\delta},z',z'_{\delta}\in\mathbb{R}^{n\times d}$ with constant $C>0.$

\textbf{(H2.2)} The fixed time delay satisfies $0 \leq \delta \leq T$, $\xi\in L^2(\Omega,\mathcal{F}_T,\mathbb{P};\mathbb{R}^n)$, the initial path of $(y,z)$: $\varphi(\cdot),\psi(\cdot)$ are given square-integrable functions and $\phi(t,s)\leq M$ is given bounded $\mathcal{F}_s-$adapted process with $0\leq s\leq t\leq T$ and $M$ is some positive constant.

\textbf{(H2.3)} $\mathbb{E}[\int^T_0|f(t,0,0,0,0)|^2dt]<+\infty$.\\

Then we have the following existence and uniqueness of the delayed BSDE (\ref{MBD02}):\begin{thm}\label{Thm1}
Suppose that (H2.1)-(H2.3) hold, then for sufficiently small time delay $\delta$, the BSDE with delay (\ref{MBD02}) has a unique adapted solution $(y(\cdot),z(\cdot))\in L^2_{\mathbb{F}}(-\delta,T;\mathbb{R}^n)\times L^2_{\mathbb{F}}(-\delta,T;\mathbb{R}^{n\times d}).$
\end{thm}\begin{proof}Let us introduce the following norm in Banach space $L^2_{\mathbb{F}}(-\delta,T;\mathbb{R}^n)$ which is equivalent to the original norm of $L^2_{\mathbb{F}}(-\delta,T;\mathbb{R}^n)$:
$$\|\nu(\cdot)\|_{\beta}=\big(\mathbb{E}[\int^T_{-\delta}|\nu(s)|^2e^{\beta s}ds]\big)^{\frac{1}{2}}.$$Set
\begin{equation}\label{MBD3}
\left\{\begin{aligned}
y(t)=&\ \xi+\int^T_tf(s,Y(s),Y_{\delta}(s),Z(s),Z_{\delta}(s))ds-\int^T_tz(s)dW(s),\,t\in[0,T],\\
y(t)=&\ \varphi(t),\,z(t)=\psi(t),\,t\in[-\delta,0).
\end{aligned}\right.\end{equation}Define a mapping $h:L^2_{\mathbb{F}}(-\delta,T;\mathbb{R}^n\times\mathbb{R}^{n\times d})\longrightarrow \mathbb{R}^n\times\mathbb{R}^{n\times d}$ such that $h[(Y(\cdot),Z(\cdot))]=(y(\cdot),z(\cdot))$. So if we can prove that $h$ is a contraction mapping under the norm $\|\cdot\|_{\beta}$, then the desired result can be obtained by the fixed point theorem. For two arbitrary elements $(Y(\cdot),Z(\cdot))$ and $(Y'(\cdot),Z'(\cdot))$ in $L^2_{\mathbb{F}}(-\delta,T;\mathbb{R}^n\times\mathbb{R}^{n\times d})$, set $(y(\cdot),z(\cdot))=h[(Y(\cdot),Z(\cdot))]$ and $(y'(\cdot),z'(\cdot))=h[(Y'(\cdot),Z'(\cdot))]$. Denote their difference by$$(\hat{Y}(\cdot),\hat{Z}(\cdot))=(Y(\cdot)-Y'(\cdot),Z(\cdot)-Z'(\cdot)),\,(\hat{y}(\cdot),\hat{z}(\cdot))=(y(\cdot)-y'(\cdot),z(\cdot)-z'(\cdot)).$$In fact Eq. (\ref{MBD3}) is a classical BSDE, and it follows that\begin{equation*}
\begin{aligned}
&\mathbb{E}[\int^T_0(\frac{\beta}{2}|\hat{y}(s)|^2+|\hat{z}(s)|^2)e^{\beta s}ds]\\
\leq &\ \frac{2}{\beta}\mathbb{E}[\int^T_0|f(s,Y(s),Y_{\delta}(s),Z(s),Z_{\delta}(s))-f(s,Y'(s),Y'_{\delta}(s),Z'(s),
Z'_{\delta}(s))|^2e^{\beta s}ds]\\
\leq &\ \frac{2C^2}{\beta}\mathbb{E}[\int^T_0\big(|\hat{Y}(s)|+|\hat{Y}_{\delta}(s)|+|\hat{Z}(s)|+|\hat{Z}_{\delta}(s)|\big)^2e^{\beta s}ds]\\
\leq &\ \frac{6C^2}{\beta}\mathbb{E}[\int^T_0\big(|\hat{Y}(s)|^2+|\hat{Z}(s)|^2+2|\hat{Y}_{\delta}(s)|^2
+2|\hat{Z}_{\delta}(s)|^2\big)e^{\beta s}ds]\\
\leq &\ \frac{6C^2}{\beta}[1+2M^2\delta\int^0_{-\delta}e^{-\beta r}\alpha(dr)]\mathbb{E}[\int^T_{-\delta}\big(|\hat{Y}(s)|^2+|\hat{Z}(s)|^2\big)e^{\beta s}ds]\\
=&\ K(C,M,\delta,\alpha,\beta)\mathbb{E}[\int^T_{-\delta}\big(|\hat{Y}(s)|^2+|\hat{Z}(s)|^2\big)e^{\beta s}ds].
\end{aligned}\end{equation*}Note that\begin{equation*}
\begin{aligned}
&\mathbb{E}\int^T_0|\hat{Y}_{\delta}(s)|^2e^{\beta s}ds\\
=&\ \mathbb{E}\int^T_0|\int^0_{-\delta}\phi(s,s+r)(Y(s+r)-Y'(s+r))\alpha(dr)|^2e^{\beta s}ds\\
\leq&\ M^2\delta\mathbb{E}\int^T_0\int^0_{-\delta}|Y(s+r)-Y'(s+r)|^2\alpha(dr)e^{\beta s}ds\\
=&\ M^2\delta\mathbb{E}\int^0_{-\delta}e^{-\beta r}\int^T_0|Y(s+r)-Y'(s+r)|^2e^{\beta (s+r)}ds\alpha(dr)\\
=&\ M^2\delta\mathbb{E}\int^0_{-\delta}e^{-\beta r}\int^{T+r}_r|Y(u)-Y'(u)|^2e^{\beta u}du\alpha(dr)\\
\leq &\ M^2\delta\mathbb{E}\int^0_{-\delta}e^{-\beta r}\alpha(dr)\int^{T}_{-\delta}|\hat{Y}(s)|^2e^{\beta s}ds.\\
\end{aligned}
\end{equation*}If we choose $\beta=\frac{1}{\delta}$, then$$K(C,M,\delta,\alpha,\beta)=6C^2\delta[1+2M^2\delta e\alpha([-\delta,0])].$$Therefore, if $\delta$ is sufficiently small satisfying $K(C,M,\delta,\alpha,\beta)<1$, then $h$ is a contraction mapping under the norm $\|\cdot\|_{\beta}$. Our proof is completed.
\end{proof}Now, let us introduce the following advanced SDE as following:\begin{equation}\label{MBD4}
\left\{\begin{aligned}
dx(t)=&\ b\left(t, x(t), \int^{t+\delta}_{t}\phi(t,s)x(s)\alpha(ds)\right)dt+\sigma\left(t, x(t), \int^{t+\delta}_{t}\phi(t,s)x(s)\alpha(ds)\right)dW(t),\,\,t\in[0,T],\\
x(0)=&\ x_0,\\
x(t)=&\ \lambda(t), \,t\in(T,T+\delta].
\end{aligned}\right.\end{equation}It is notable that there exist some results to discuss the time-advanced ordinary differential equations (AODEs) (e.g., refer \cite{alw}, \cite{hwg}, \cite{km}, \cite{or}, \cite{po}, \cite{y}, etc.) which have been applied in various areas including traveling waves in physics, cell-growth in population dynamics, capital market in economics, life-cycle models, electronics, etc. However, to our best knowledge, the stochastic differential equations of advanced type (ASDE) has never been discussed before. Nevertheless, these stochastic advanced equations should also have considerable real meanings besides the control study only (as implied by the broad-range application of AODES, their deterministic counterpart). Keep this in mind, we will discuss these meanings in future study. Now we aim to study the $\mathcal{F}_t$-adapted solution $x(\cdot)\in L^2_{\mathbb{F}}(0,T+\delta;\mathbb{R}^n)$ of the ASDE (\ref{MBD4}). Suppose that for all $t \in [0,T],$ $b:\Omega\times\mathbb{R}^n\times L^2(\Omega,\mathcal{F}_r,\mathbb{P};\mathbb{R}^n)\rightarrow L^2(\Omega,\mathcal{F}_t,\mathbb{P};\mathbb{R}^n)$, $\sigma:\Omega\times\mathbb{R}^n\times L^2(\Omega,\mathcal{F}_r,\mathbb{P};\mathbb{R}^n)\rightarrow L^2(\Omega,\mathcal{F}_t,\mathbb{P};\mathbb{R}^{n\times d})$, where $r\in[t,T+\delta]$. We also assume that $b$ and $\sigma$ satisfies the following conditions:

\textbf{(H2.4)} There exists a constant $C>0$, such that for all $t\in[0,T],$ $x,x'\in\mathbb{R}^n$, $\zeta(\cdot),\zeta'(\cdot)\in L^2_{\mathbb{F}}
(t,T+\delta;\mathbb{R}^n)$, $r\in[t,T+\delta],$ we have
\begin{equation*}
\begin{aligned}
&|b(t,x,\zeta(r))-b(t,x',\zeta'(r))|+|\sigma(t,x,\zeta(r))-\sigma(t,x',\zeta'(r))|\\
\leq & \ C(|x-x'|+\mathbb{E}^{\mathcal{F}_t}[|\zeta(r)-\zeta'(r)|]).
\end{aligned}
\end{equation*}

\textbf{(H2.5)} $$\sup_{0\leq t\leq T}\big(|b(t,0,0)+\sigma(t,0,0)|\big)<+\infty.$$
Under these conditions, $b(t,\cdot,\cdot)$ and $\sigma(t,\cdot,\cdot)$ are $\mathcal{F}_t$-measurable and this ensures the solution of the advanced SDE will be $\mathcal{F}_t$-adapted. We have the following result to the ASDE (\ref{MBD4}).\begin{thm}\label{Thm2}
Assume $b$ and $\sigma$ satisfy (H2.4) and (H2.5), $\mathbb{E}|x_0|^2<+\infty$, $\mathbb{E}\sup_{T\leq t\leq T+\delta}|\lambda(t)|^2<+\infty,$ and the time delay $\delta$ is sufficiently small, then the ASDE (\ref{MBD4}) admits a unique $\mathcal{F}_t$-adapted solution.
\end{thm}\begin{proof}Similar to Theorem \ref{Thm1}, let us define the following norm in Banach space $L^2_{\mathbb{F}}(0,T+\delta;\mathbb{R}^n)$ which is more convenient for us to construct a contraction mapping:
$$\|\nu(\cdot)\|_{\beta}=\big(\mathbb{E}[\int^{T+\delta}_0|\nu(s)|^2e^{-\beta s}ds]\big)^{\frac{1}{2}}.$$
For simplicity, we denote $\int^{t+\delta}_{t}\phi(t,s)x(s)\alpha(ds)$ by $x_{\delta^+}(t)$, and set\begin{equation*}
\left\{\begin{aligned}
x(t)=&\ x_0+\int^t_0b(s,X(s),X_{\delta^+}(s))ds+\int^t_0\sigma(s,X(s),X_{\delta^+}(s))dW(s),\,t\in[0,T],\\
x(t)=&\ \lambda(t),\,t\in(T,T+\delta].
\end{aligned}
\right.
\end{equation*}Then we can define a mapping $I:L^2_{\mathbb{F}}(0,T+\delta;\mathbb{R}^n)\rightarrow L^2_{\mathbb{F}}(0,T+\delta;\mathbb{R}^n)$ such that $I[X(\cdot)]=x(\cdot)$. For arbitrary $X(\cdot),X'(\cdot)\in L^2_{\mathbb{F}}(0,T+\delta;\mathbb{R}^n)$, we introduce the following notations:$$\begin{aligned}
I[X(\cdot)]=x(\cdot)&,\,\,\,I[X'(\cdot)]=x'(\cdot),\\
\hat{X}(\cdot)=X(\cdot)-X'(\cdot)&,\,\,\,\hat{x}(\cdot)=x(\cdot)-x'(\cdot).
\end{aligned}$$Consequently, $\hat{x}(\cdot)$ satisfies\begin{equation*}
\left\{\begin{aligned}
\hat{x}(t)=&\ \int^t_0[b(s,X(s),X_{\delta^+}(s))-b(s,X'(s),X'_{\delta^+}(s))]ds\\
&\ +\int^t_0[\sigma(s,X(s),X_{\delta^+}(s))-\sigma(s,X'(s),X'_{\delta^+}(s))]dW(s),\,\,t\in[0,T],\\
\hat{x}(0)=&\ 0,\\
\hat{x}(t)=&\ 0,\,\,t\in(T,T+\delta].
\end{aligned}\right.\end{equation*}Applying It\^{o}'s formula to $e^{-\beta t}|\hat{x}(t)|^2$ on $[0,T]$, we get\begin{equation*}\begin{aligned}
&\ \mathbb{E}[e^{-\beta T}|\hat{x}(T)|^2]+\beta\mathbb{E}[\int^T_0 e^{-\beta t}|\hat{x}(t)|^2dt]\\
=&\ \mathbb{E}[\int^T_0(2e^{-\beta t}\langle\hat{b}(t),\hat{x}(t)\rangle+e^{-\beta t}\langle\hat{\sigma}(t),\hat{\sigma}(t)\rangle)dt],
\end{aligned}
\end{equation*}with$$\begin{aligned}
\hat{b}(t)=&\ b(t,X(t),X_{\delta^+}(t))-b(t,X'(t),X'_{\delta^+}(t)),\\
\hat{\sigma}(t)=&\ \sigma(t,X(t),X_{\delta^+}(t))-\sigma(t,X'(t),X'_{\delta^+}(t)).
\end{aligned}$$Since $b, \sigma$ satisfy (H2.4), we have\begin{equation*}
\begin{aligned}
&\ \beta\mathbb{E}[\int^T_0e^{-\beta t}|\hat{x}(t)|^2dt]\\
\leq &\ \mathbb{E}[\int^T_0e^{-\beta t}|\hat{x}(t)|^2dt]+\mathbb{E}[\int^T_0e^{-\beta t}|\hat{b}(t)|^2dt]+\mathbb{E}[\int^T_0e^{-\beta t}|\hat{\sigma}(t)|^2dt],\\
\leq &\ \mathbb{E}[\int^T_0e^{-\beta t}|\hat{x}(t)|^2dt]+2C^2\mathbb{E}\Big[\int^T_0e^{-\beta t}\big(|\hat{X}(t)|+\mathbb{E}^{\mathcal{F}_t}[|\hat{X}_{\delta^+}(t)|]\big)^2dt\Big].
\end{aligned}
\end{equation*}Moreover, it follows that\begin{equation*}
\begin{aligned}
&\ (\beta-1)\mathbb{E}[\int^T_0e^{-\beta t}|\hat{x}(t)|^2dt]\\
\leq &\ 4C^2\mathbb{E}[\int^T_0e^{-\beta t}|\hat{X}(t)|^2dt]+4C^2\mathbb{E}[\int^T_0e^{-\beta t}|\hat{X}_{\delta^+}(t)|^2dt]\\
\leq &\ 4C^2[1+M^2\delta\int^{\delta}_0e^{\beta s}\alpha(ds)]\mathbb{E}[\int^{T+\delta}_0e^{-\beta t}|\hat{X}(t)|^2dt],
\end{aligned}
\end{equation*}due to the fact$$\begin{aligned}
&\ \mathbb{E}[\int^T_0e^{-\beta t}|\hat{X}_{\delta^+}(t)|^2dt]\\
=&\ \mathbb{E}[\int^{T}_{0}e^{-\beta t}|\int^{t+\delta}_t\phi(t,s)\hat{X}(s)\alpha(ds)|^2dt]\\
\leq&\ M^2\delta\mathbb{E}[\int^{T}_0e^{-\beta t}\int^{\delta}_0|\hat{X}(s)|^2\alpha(ds)dt]\\
=&\ M^2\delta\mathbb{E}[\int^{\delta}_0e^{\beta s}\int^{T}_0e^{-\beta(t+s)}|\hat{X}(s)|^2dt\alpha(ds)]\\
\leq&\ M^2\delta\int^{\delta}_0e^{\beta s}\alpha(ds)\mathbb{E}[\int^{T+\delta}_0e^{-\beta t}|\hat{X}(t)|^2dt].
\end{aligned}$$Set$$K'(C,M,\delta,\alpha,\beta)=\frac{4C^2[1+M^2\delta\int^{\delta}_0e^{\beta s}\alpha(ds)]}{\beta-1}.$$If we choose $\beta=\frac{1}{\delta}$, then for sufficiently small $\delta$, we have $K'(C,M,\delta,\alpha,\beta)\leq
 \frac{4C^2\delta[1+M^2\delta e\alpha([0,\delta])]}{1-\delta}<1$. It follows the mapping $I$ is contraction, hence the result.\end{proof}

\section{Optimal control problem for backward stochastic system with delay}

In this section we study a kind of stochastic recursive delayed control problems as follows:\begin{equation}\label{MBD6}
\left\{\begin{aligned}
-dy(t)=&\ f\left(t,y(t),\int^t_{t-\delta}\phi(t,s)y(s)\alpha(ds),z(t),\int^t_{t-\delta}\phi(t,s)z(s)\alpha(ds), v(t),\int^t_{t-\delta}\phi(t,s)v(s)\alpha(ds)\right)dt\\&\ -z(t)dW(t),\,\quad t\in[0,T],\\
y(T)=&\ \xi,\\
y(t)=&\ \varphi(t),\,z(t)=\psi(t),\,\quad t\in[-\delta,0).\end{aligned}\right.
\end{equation}Here $f:\Omega\times[0,T]\times\mathbb{R}^n\times\mathbb{R}^n\times\mathbb{R}^{n\times d}\times\mathbb{R}^{n\times d}\times\mathbb{R}^k\times\mathbb{R}^k\longrightarrow
\mathbb{R}^n$ is given measurable function, $\xi\in L^2(\Omega,\mathcal{F}_T,\mathbb{P};\mathbb{R}^n)$, $\varphi(\cdot)$ is deterministic function. $v(\cdot)$ is the control process with initial path $\eta$. The stochastic recursive control problems is to find the optimal control to achieve a pre-given goal $\xi$ at the terminal time $T$, and also maximize some given cost functional. Let $U$ be a nonempty convex subset. We denote $\mathcal{U}$ the set of all admissible control processes $v(\cdot)$ of the form$$v(t)=\left\{\begin{aligned}
\eta(t)&,\,\,\,t\in[-\delta,0),\\
v(t)\in & L^2_{\mathbb{F}}(0,T;\mathbb{R}^k),\,v(t)\in U,\,a.s.,\,t\in[0,T].
\end{aligned}\right.$$The objective is to maximize the following functional over $\mathcal{U}$:
$$\begin{aligned}
J(v(\cdot)) =&\ \mathbb{E}[\int_0^{T}l \left(t, y(t), \int_{t-\delta}^{t}\phi(t,s)y(s)\alpha(ds), z(t), \int_{t-\delta}^{t}\phi(t,s)z(s)\alpha(ds),  v(t),\int_{t-\delta}^{t}\phi(t,s)v(s)\alpha(ds)\right)dt\\
&\ +\gamma(y(0))].\\
\end{aligned}$$For simplicity, denote $(\int^t_{t-\delta}\phi(t,s)y(s)\alpha(ds), \int^t_{t-\delta}\phi(t,s)z(s)\alpha(ds), \int^t_{t-\delta}\phi(t,s)v(s)\alpha(ds))$ by $(y_{\delta}(t),$ $z_{\delta}(t), v_{\delta}(t))$ if no confusion occurs.

\textbf{(H3.1)} $f$ is continuously differentiable in $(y, y_{\delta}, z, z_{\delta}, v, v_{\delta})$. Moreover, the partial derivatives $f_y, f_{y_{\delta}}, f_z, f_{z_{\delta}}, f_v$ and $f_{v_{\delta}}$ of $f$ with respect to $(y,y_{\delta}, z, z_{\delta}, v, v_{\delta})$ are uniformly bounded.

Then if $v(\cdot)$ is admissible control and assumption (H3.1) holds, then the delayed BSDE (\ref{MBD6}) has a unique solution $(y^v(\cdot),z^v(\cdot))\in L^2_{\mathbb{F}}(0,T+\delta;\mathbb{R}^n)\times L^2_{\mathbb{F}}(0,T+\delta;\mathbb{R}^{n\times d})$ on $[0,T+\delta]$ for sufficiently small $0\leq\delta\leq T$.

\textbf{(H3.2)} For each $v(\cdot)\in\mathcal{U}$, $l(\cdot,y^v(\cdot),y^v_{\delta}(\cdot),z^v(\cdot),z^v_{\delta}(\cdot),v(\cdot),v_{\delta}(\cdot))\in L^1_{\mathbb{F}}(0,T;\mathbb{R})$, $l$ is differentiable to $(y,y_{\delta},z,z_{\delta},v,v_{\delta}),$ $\gamma$ is differentiable with respect to $y$, and all the derivatives are bounded.

Define the Hamiltonian function $H:[0,T]\times\mathbb{R}^n\times\mathbb{R}^n\times\mathbb{R}^{n\times d}\times\mathbb{R}^k\times\mathbb{R}^k\times\mathbb{R}^n\rightarrow\mathbb{R}$ by\begin{equation*}
\begin{aligned}
H(t,y,y_{\delta},z,z_{\delta},v,v_{\delta})
=\ l(t,y,y_{\delta},z,z_{\delta},v,v_{\delta})-\langle f(t,y,y_{\delta},z,z_{\delta},v,v_{\delta}),p \rangle.
\end{aligned}
\end{equation*}For each $v(\cdot)\in\mathcal{U},$ the associated adjoint equation satisfies the following ASDE:
\begin{equation}\label{MBD7}
\left\{\begin{aligned}
dp^v(t)=&\ \big\{ -H_y(t,\Theta^v(t),v(t),v_{\delta}(t),p^v(t))\\
&\ -\mathbb{E}^{\mathcal{F}_t}[\int^{t+\delta}_tH_{y_{\delta}}(s,\Theta^v(s),v(s),v_{\delta}(s),p^v(s))\phi(s,t)
\chi_{[0,T]}(s)ds]\frac{\alpha(dt)}{dt}\big\}dt\\
&+\big\{-H_z(t,\Theta^v(t),v(t),v_{\delta}(t),p^v(t))\\
&\ -\mathbb{E}^{\mathcal{F}_t}[\int^{t+\delta}_tH_{z_{\delta}}(s,\Theta^v(s),v(s),v_{\delta}(s),p^v(s))\phi(s,t)
\chi_{[0,T]}(s)ds]\frac{\alpha(dt)}{dt}\big\}dW(t),\,\,t\in[0,T],\\
p^v(0)=&\ -\gamma_y(y(0)),\\
\end{aligned}\right.
\end{equation}with $\Theta^v(t)=(y^v(t),y^v_{\delta}(t),z^v(t),z^v_{\delta}(t))$ and $\frac{\alpha(dt)}{dt}$ is the Radon-Nikodym derivative.\begin{rmk}For a given admissible control $v(\cdot)$, Eq.(\ref{MBD7}) is an ASDE. By the virtue of the indicative function $\chi_{[0,T]}(s)$, it is not necessary to give the value of $p^v(t)$ on $(T,T+\delta]$. Moreover, the ASDE (\ref{MBD7}) admits a unique solution under condition (H3.1) and (H3.2) due to Theorem \ref{Thm2}.\end{rmk}Now we can give the first main result of this paper in the following:\begin{thm}\label{Thm3} (Sufficient condition of optimality) Let (H3.1) and (H3.2) hold. Suppose for $u(\cdot)\in\mathcal{U}$, $(y(\cdot), z(\cdot))$ is the corresponding trajectory and $p(\cdot)$ the corresponding solution of adjoint equation (\ref{MBD7}). If the following condition holds true: \begin{equation}\label{MBD10}
 \begin{aligned}
 &\langle H_v(t,\Theta(t),u(t),u_{\delta}(t),p(t))\\
 &\ +\mathbb{E}^{\mathcal{F}_t}
 [\int^{t+\delta}_tH_{v_{\delta}}(s,\Theta(s),u(t),u_{\delta}(s),p(s))\phi(s,t)\chi_{[0,T]}(s)ds]
 \frac{\alpha(dt)}{dt},u(t)\rangle\\
 =&\ \max_{v\in U}\langle H_v(t,\Theta(t),u(t),u_{\delta}(t),p(t))\\
 &\ +\mathbb{E}^{\mathcal{F}_t}
 [\int^{t+\delta}_tH_{v_{\delta}}(s,\Theta(s),u(t),u_{\delta}(s),p(s))\phi(s,t)\chi_{[0,T]}(s)ds]
 \frac{\alpha(dt)}{dt},v\rangle,
 \end{aligned}
\end{equation}moreover, if $H(t,y,y_{\delta},z,z_{\delta},v,v_{\delta},p(t))$ is a concave function of $(y, y_{\delta}, z, z_{\delta}, v, v_{\delta})$, and $\gamma$ is concave in $y$, then $u(\cdot)$ is an optimal control for our problem.\end{thm}
\begin{proof}Choose a $v(\cdot)\in\mathcal{U}$ and let $(y^v(\cdot),z^v(\cdot))$ be the corresponding solution of (\ref{MBD6}). To simplify the notation, we also use $\Theta^v(t)=(y^v(t),y^v_{\delta}(t),z^v(t),z^v_{\delta}(t))$ and $\Theta(t)=(y(t),y_{\delta}(t),z(t),z_{\delta}(t))$. Let\begin{equation*}
\begin{aligned}
I=&\ \mathbb{E}\big[\int^T_0\left\{l(t,y(t),y_{\delta}(t),z(t),z_{\delta}(t),u(t),u_{\delta}(t))
-l(t,y^v(t),y^v_{\delta}(t),z^v(t),z^v_{\delta}(t),v(t),v_{\delta}(t))\right\}dt\big],\\
II=&\ \big[\gamma(y(0))-\gamma(y^v(0))\big].
\end{aligned}\end{equation*}We want to prove that\begin{equation}\label{MBD9}
J(u(\cdot))-J(v(\cdot))=I+II\geq0.
\end{equation}Since $\gamma$ is concave on $y$,
$$II\geq\gamma_y(y(0))^{\top}(y(0)-y^v(0))=-p(0)^{\top}(y(0)-y^v(0)).$$Applying It\^{o}'s formula to $\langle p(\cdot),y(\cdot)-y^v(\cdot)\rangle$, we have\begin{equation}\label{MBD11}
\begin{aligned}
&\ p(0)^{\top}(y(0)-y^v(0))\\
=&\mathbb{E}\int^T_0 \langle p(t),f(t,\Theta(t),u(t),u_{\delta}(t))-f(t,\Theta^v(t),v(t),v_{\delta}(t))\rangle dt\\
&+\ \mathbb{E}\int^T_0 \langle H_y(t,\Theta(t),u(t),u_{\delta}(t),p(t))\\
&\ +\mathbb{E}^{\mathcal{F}_t}[\int^{t+\delta}_tH_{y_{\delta}}(s,\Theta(s),u(s),u_{\delta}(s),p(s))\phi(s,t)
\chi_{[0,T]}(s)ds]
\frac{\alpha(dt)}{dt},y(t)-y^v(t) \rangle dt\\
&+ \mathbb{E}\int^T_0 \langle H_z(t,\Theta(t),u(t),u_{\delta}(t),p(t))\\
&\ +\mathbb{E}^{\mathcal{F}_t}[\int^{t+\delta}_tH_{z_{\delta}}(s,\Theta(s),u(s),u_{\delta}(s),p(s))\phi(s,t)
\chi_{[0,T]}(s)ds]
\frac{\alpha(dt)}{dt},z(t)-z^v(t)\rangle dt.\\
\end{aligned}
\end{equation}On the other hand,\begin{equation}\label{MBD12}
\begin{aligned}
I=&\ \mathbb{E}\int^T_0[H(t,\Theta(t),u(t),u_{\delta}(t),p(t))-H(t,\Theta^v(t),v(t),v_{\delta}(t),p(t))]dt\\
&\ +\mathbb{E}\int^T_0\langle p(t),f(t,\Theta(t),u(t),u_{\delta}(t))-f(t,\Theta^v(t),v(t),v_{\delta}(t))\rangle dt.
\end{aligned}\end{equation}Since $(\Theta,v,v_{\delta})\rightarrow H(t,\Theta,v,v_{\delta},p(t))$ is concave, we have
\begin{equation}\label{MBD13}
\begin{aligned}
I\geq&\ -\mathbb{E}\int^T_0\langle H_y(t,\Theta(t),u(t),u_{\delta}(t),p(t)),y^v(t)-y(t)\rangle dt\\
&\ -\mathbb{E}\int^T_0\langle H_{y_{\delta}}(t,\Theta(t),u(t),u_{\delta}(t),p(t)),y^v_{\delta}(t)-y_{\delta}(t)\rangle dt\\&\ -\mathbb{E}\int^T_0\langle H_z(t,\Theta(t),u(t),u_{\delta}(t),p(t)),z^v(t)-z(t)\rangle dt\\
&\ -\mathbb{E}\int^T_0\langle H_{z_{\delta}}(t,\Theta(t),u(t),u_{\delta}(t),p(t)),z^v_{\delta}(t)-z_{\delta}(t)\rangle dt\\
&\ -\mathbb{E}\int^T_0\langle H_v(t,\Theta(t),u(t),u_{\delta}(t),p(t)),v(t)-u(t)\rangle dt\\
&\ -\mathbb{E}\int^T_0\langle H_{v_{\delta}}(t,\Theta(t),u(t),u_{\delta}(t),p(t)),v_{\delta}(t)-u_{\delta}(t)\rangle dt\\
&\ +\mathbb{E}\int^T_0\langle p(t),f(t,\Theta(t),u(t),u_{\delta}(t))-f(t,\Theta^v(t),v(t),v_{\delta}(t))\rangle dt.
\end{aligned}
\end{equation}Moreover, we have\begin{equation}\label{MBD14}\begin{aligned}
&\ \mathbb{E}\int^T_0\langle H_{v_{\delta}}(t,\Theta(t),u(t),u_{\delta}(t),p(t)),v_{\delta}(t)-u_{\delta}(t)\rangle dt\\
=&\ \mathbb{E}\int^T_0\langle H_{v_{\delta}}(s,\Theta(s),u(s),u_{\delta}(s),p(s)),\int^s_{s-\delta}\phi(s,r)(v(r)-u(r))
\alpha(dr)\rangle ds\\
=&\ \mathbb{E}\int^T_0\langle \mathbb{E}^{\mathcal{F}_r}\int^{r+\delta}_rH_{v_{\delta}}(s,\Theta(s),u(s),
u_{\delta}(s),p(s))\phi(s,r)\chi_{[0,T]}(s)ds,
v(r)-u(r)\rangle \alpha(dr)\\
=&\ \mathbb{E}\int^T_0\langle \mathbb{E}^{\mathcal{F}_t}[\int^{t+\delta}_tH_{v_{\delta}}(s,\Theta(s),u(s),u_{\delta}(s),p(s))\phi(s,t)\chi_{[0,T]}(s)ds]\frac{\alpha(dt)}{dt},
v(t)-u(t)\rangle dt.\\
\end{aligned}\end{equation}By the maximum condition (\ref{MBD10}), we can obtain\begin{equation}\label{MBD014}
\begin{aligned}
&\ \mathbb{E}\int^T_0\langle H_v(t,\Theta(t),u(t),u_{\delta}(t),p(t)),v(t)-u(t)\rangle dt\\
&\ +\mathbb{E}\int^T_0\langle H_{v_{\delta}}(t,\Theta(t),u(t),u_{\delta}(t),p(t)),v_{\delta}(t)-u_{\delta}(t)\rangle dt\\
=&\ 0.
\end{aligned}\end{equation}From (\ref{MBD9})-(\ref{MBD014}), it is easy to get\begin{equation*}
\begin{aligned}
&\ J(u(\cdot)-J(v(\cdot))\\
\geq&\ -\mathbb{E}\int^T_0\langle H_y(t,\Theta(t),u(t),u_{\delta}(t),p(t)),y^v(t)-y(t)\rangle dt\\
&\ -\mathbb{E}\int^T_0\langle H_{y_{\delta}}(t,\Theta(t),u(t),u_{\delta}(t),p(t)),y^v_{\delta}(t)-y_{\delta}(t)\rangle dt\\
&\ -\mathbb{E}\int^T_0\langle H_z(t,\Theta(t),u(t),u_{\delta}(t),p(t)),z^v(t)-z(t)\rangle dt\\
&\ -\mathbb{E}\int^T_0\langle H_{z_{\delta}}(t,\Theta(t),u(t),u_{\delta}(t),p(t)),z^v_{\delta}(t)-z_{\delta}(t)\rangle dt\\
&+\ \mathbb{E}\int^T_0\langle H_y(t,\Theta(t),u(t),u_{\delta}(t),p(t))\\
&\ +\mathbb{E}^{\mathcal{F}_t}[\int^{t+\delta}_tH_{y_{\delta}}
(s,\Theta(s),u(s),u_{\delta}(s),p(s))\phi(s,t)\chi_{[0,T]}(s)ds]
\frac{\alpha(dt)}{dt},y^v(t)-y(t)\rangle dt\\
&+ \mathbb{E}\int^T_0\langle H_z(t,\Theta(t),u(t),u_{\delta}(t),p(t))\\
&\ +\mathbb{E}^{\mathcal{F}_t}[\int^{t+\delta}_{t}H_{z_{\delta}}
(s,\Theta(s),u(s),u_{\delta}(s),p(s))\phi(s,t)\chi_{[0,T]}(s)ds]
\frac{\alpha(dt)}{dt},z^v(t)-z(t)\rangle dt\\
=&\ 0.
\end{aligned}\end{equation*}So, we verify that $J(u(\cdot))-J(v(\cdot))\geq0$ for any $v(\cdot)\in\mathcal{U}$, and it follows that $u(\cdot)$ is the optimal control.\end{proof}\begin{cor}\label{Cor1}If the $\alpha(dt)$ is the Dirac measure at $-\delta$, then the system involves pointwise delay, i.e. $y_{\delta}(t)=y(t-\delta),z_{\delta}(t)=z(t-\delta),v_{\delta}(t)=v(t-\delta)$. In this case, the sufficient condition of optimality is$$H_v(t,\Theta(t),u(t),u(t-\delta),p(t))+\mathbb{E}^{\mathcal{F}_t}[H_{v_{\delta}}(t+\delta,\Theta(t+\delta),u(t),u(t+\delta),p(t+\delta))]=0,$$
with adjoint equation\begin{equation}\label{MBD07}
\left\{\begin{aligned}
dp^v(t)=&\ \big\{ -H_y(t,\Theta^v(t),v(t),v(t-\delta),p^v(t))\\
&\ -\mathbb{E}^{\mathcal{F}_t}[H_{y_{\delta}}(t+\delta,\Theta^v(t+\delta),v(t+\delta),v(t),p^v(t+\delta))]\big\}dt\\
&\big\{  -H_z(t,\Theta^v(t),v(t),v(t-\delta),p^v(t))\\
&\ -\mathbb{E}^{\mathcal{F}_t}[H_{z_{\delta}}(t+\delta,\Theta^v(t+\delta),v(t+\delta),v(t),p^v(t+\delta))]\big\}dW(t),\,\,t\in[0,T],\\
p^v(0)=&\ -\gamma_y(y(0)),\\
p^v(t)=&\ 0,\,\,t\in(T,T+\delta],
\end{aligned}\right.
\end{equation}where $\Theta^v(t)=(y^v(t),y^v(t-\delta),z^v(t),z^v(t-\delta))$.
\end{cor}Now we consider the special case wherein the control variable involves no delay, to derive the corresponding maximum condition, we first introduce the following condition.

\textbf{(H3.3)} For each $v(\cdot)\in\mathcal{U}$, $l(\cdot,y^v(\cdot),y^v_{\delta}(\cdot),z^v(\cdot),z^v_{\delta}(\cdot),v(\cdot))\in L^1_{\mathbb{F}}(0,T;\mathbb{R})$, $l$ is differentiable on $(y,y_{\delta},z,z_{\delta})$ and $\gamma$ is differentiable with respect to $y,$ all  derivatives are bounded.

We have the following result:\begin{thm}(The case without control delay)\label{Thm4}
In case there has no control delay, that is, $$f=f(\cdot,y^v(\cdot),y^v_{\delta}(\cdot),z^v(\cdot),z^v_{\delta}(\cdot),v(\cdot)), \\\quad l=l(\cdot,y^v(\cdot),y^v_{\delta}(\cdot),z^v(\cdot),z^v_{\delta}(\cdot),v(\cdot)).$$ Suppose $u(\cdot)\in\mathcal{U}$, $(y(\cdot),z(t))$ is its corresponding trajectory and $p(\cdot)$ the corresponding solution of adjoint equation (\ref{MBD7}). Let (H3.1), (H3.3) and the following condition holds true:\begin{equation}\label{MBD8}
 H(t,\Theta(t),u(t),p(t))= \max_{v\in U}H(t,\Theta(t),v,p(t)),\,\,\textrm{for all}\,t\in[0,T],
\end{equation}with $\Theta(t)=(y(t),y_{\delta}(t),z(t),z_{\delta}(t)),$ moreover, suppose for each $(t,y,y_{\delta},z,z_{\delta})\in[0,T]\times\mathbb{R}^n\times\mathbb{R}^n\times\mathbb{R}^{n\times d}\times\mathbb{R}^{n\times d},$ $\hat{H}(t,y,y_{\delta},z,z_{\delta})=\max_{v\in U}H(t,y,y_{\delta},z,z_{\delta},v,p(t))$ is a concave function of $(y,y_{\delta},z,z_{\delta})$, and $\gamma$ is concave in $y$, then $u(\cdot)$ is an optimal control.
\end{thm}\begin{proof}Similar to the proof of Theorem \ref{Thm3}, we also choose arbitrary $v(\cdot)\in \mathcal{U}$, and aim to prove $J(u(\cdot))-J(v(\cdot))\geq 0.$ From the procedure of Theorem \ref{Thm3}, we can see that\begin{equation}\label{MBD16}
\begin{aligned}
&\ J(u(\cdot)-J(v(\cdot))\\
\geq&\ \mathbb{E}\int^T_0[H(t,\Theta(t),u(t),p(t))-H(t,\Theta^v(t),v(t),p(t))]dt\\
&\ -\mathbb{E}\int^T_0\langle H_y(t,\Theta(t),u(t),p(t))\\
&\ +\mathbb{E}^{\mathcal{F}_t}[\int^{t+\delta}_tH_{y_{\delta}}(s,\Theta(s),u(s),p(s))\phi(s,t)\chi_{[0,T]}(s)ds]
\frac{\alpha(dt)}{dt},y(t)-y^v(t)\rangle dt\\
&- \mathbb{E}\int^T_0\langle H_z(t,\Theta(t),u(t),p(t))\\
&\ +\mathbb{E}^{\mathcal{F}_t}[\int^{t+\delta}_tH_{z_{\delta}}(s,\Theta(s),u(s),p(s))\phi(s,t)\chi_{[0,T]}(s)ds]
\frac{\alpha(dt)}{dt},z(t)-z^v(t)\rangle dt.\\
\end{aligned}
\end{equation}By the condition (\ref{MBD8}) and the definition of $\hat{H},$
\begin{equation}\label{MBD17}
\begin{aligned}
&\ H(t,\Theta(t),u(t),p(t))-H(t,y,y_{\delta},z,z_{\delta},v,p(t))\\
\geq&\ \hat{H}(t,\Theta(t))-\hat{H}(t,y,y_{\delta},z,z_{\delta}).
\end{aligned}\end{equation}Since $(y,y_{\delta},z,z_{\delta})\rightarrow\hat{H}(t,y,y_{\delta},z,z_{\delta})$ is concave for any given $t\in[0,T]$, it follows that there exists a supergradient $a_1(t),a_2(t)\in\mathbb{R}^n$ and $b_1(t),b_2(t)\in\mathbb{R}^{n\times d}$ for $\hat{H}(t,y,y_{\delta},z,z_{\delta})$ at $(y,y_{\delta},z,z_{\delta})$ (refer Chapter 5, Section 23 in \cite{R}), that is, for all $(y,y_{\delta},z,z_{\delta}),$\begin{equation}\label{MBD18}
\begin{aligned}
&\ \hat{H}(t,y,y_{\delta},z,z_{\delta})-\hat{H}(t,\Theta(t))\\
\leq&\ \langle a_1(t),y-y(t)\rangle+\langle a_2(t),y_{\delta}-y_{\delta}(t)\rangle+\langle b_1(t),z-z(t)\rangle+\langle b_2(t),z_{\delta}-z_{\delta}(t)\rangle,
\end{aligned}\end{equation}Define
\begin{equation*}
\begin{aligned}
\Gamma(t,y,y_{\delta},z,z_{\delta})=&\ H(t,y,y_{\delta},z,z_{\delta},u(t),p(t))-H(t,\Theta(t),u(t),p(t))\\
&\ -\langle a_1(t),y-y(t)\rangle-\langle a_2(t),y_{\delta}-y_{\delta}(t)\rangle\\
&\ -\langle b_1(t),z-z(t)\rangle-\langle b_2(t),z_{\delta}-z_{\delta}(t)\rangle.
\end{aligned}\end{equation*}Obviously, $\Gamma(t,y,y_{\delta},z,z_{\delta})\leq0$ for all $(y,y_{\delta},z,z_{\delta})$ and $\Gamma(t,\Theta(t))=0$. It follows that $\Gamma$ attains its maximum value at $(y(t),y_{\delta}(t),z(t),z_{\delta}(t))$. Consequently we have
$$\begin{aligned}
\Gamma_y(t,\Theta(t))=0,\,\,\Gamma_{y_{\delta}}(t,\Theta(t))=0,\\
\Gamma_z(t,\Theta(t))=0,\,\,\Gamma_{z_{\delta}}(t,\Theta(t))=0.\\
\end{aligned}$$These will lead to$$\begin{aligned}
H_y(t,\Theta(t),u(t),p(t))=a_1(t),\,\,H_{y_{\delta}}(t,\Theta(t),u(t),p(t))=a_2(t),\\
H_z(t,\Theta(t),u(t),p(t))=b_1(t),\,\,H_{z_{\delta}}(t,\Theta(t),u(t),p(t))=b_2(t).\\
\end{aligned}$$Combine (\ref{MBD18}) and note the arbitrariness of $(y,y_{\delta},z,z_{\delta})$, we have
\begin{equation*}\begin{aligned}&\ \hat{H}(t,\Theta^v(t))-\hat{H}(t,\Theta(t))\\
\leq&\ \langle H_y(t,\Theta(t),u(t),p(t)),y^v(t)-y(t)\rangle+\langle H_{y_{\delta}}(t,\Theta(t),u(t),p(t)),y^v_{\delta}(t)-y_{\delta}(t)\rangle\\
&\ +\langle H_z(t,\Theta(t),u(t),p(t)),z^v(t)-z(t)\rangle+\langle H_{z_{\delta}}(t,\Theta(t),u(t),p(t)),z^v_{\delta}(t)-z_{\delta}(t)\rangle.
\end{aligned}\end{equation*}Substitute the above result into (\ref{MBD16}), we obtain $J(u(\cdot))-J(v(\cdot))\geq0.$
\end{proof}

\section{Application I: Dynamic optimization of recursive utility with moving average}
In this section, we investigate Example $1.1$: the dynamic optimization of recursive utility with moving average, which is already given in Section $1$. The state equation satisfies the following dynamics:\begin{equation}\label{DMP4}
y(t)=\xi-\int_t^{T}\left[\alpha c(s)+\beta \left(\frac{1}{s}\int_0^{s}y(u)du\right)\right]ds-\int_t^{T}z(s)dW_s
\end{equation}where $\alpha,\beta>0$ are some constants, and the control variable is consumption process $c(\cdot).$ The class of admissible controls is denoted by $\mathcal{C}=\{c(\cdot)\in L^2_{\mathbb{F}}(0,T;\mathbb{R}),t\in[0,T]\}$. Given some standard utility function $U$, for example,  $U(x)=\frac{x^{R}}{R}$ for $0<R<1,$ we can consider the following dynamic optimization problem: $$\inf_{c (\cdot) \in \mathcal{C}}J(c(\cdot))$$where the objective functional is given by $$J(c(\cdot))=-\mathbb{E}[\int^T_0U(c(t))dt]+y^{c}(0)$$which follows Delong \cite{d1}. The state of Eq.(\ref{DMP4}) can be reformulated as$$y(t)=\xi-\int_t^{T}\left[\alpha c(s)+\beta \left(\frac{1}{s}\int^s_{s-T}Ty(u)\chi_{\{u\geq0\}}\alpha(du)\right)\right]ds-\int_t^{T}z(s)dW_s$$where $\alpha$ is the uniform measure. Introduce the Hamiltonian function$$H(t,y(t),y_{\delta}(t),z(t),c(t),p(t))=-U(c(t))+\left[\alpha c(t)+\beta \left(\frac{1}{t}\int^t_{t-T}Ty(u)\chi_{\{u\geq0\}}(u)\alpha(du)\right)\right]p(t).$$
The associated adjoint equation satifies\begin{equation}\label{MBD25}\left\{\begin{aligned}
dp(t)=&\ \left(\int^T_t\beta p(s)\frac{1}{s}ds\right)dt,\,\,t\in[0,T],\\
p(0)=&\ 1.\end{aligned}\right.\end{equation}It follows Eq.(\ref{MBD25}) can be reduced to the following ordinary differential equation:$$\dot{p}(t)=\int^T_t\beta p(s)\frac{1}{s}ds,\,\,\ddot{p}(t)=-\frac{\beta}{t}p(t)$$which is solvable and by Theorem \ref{Thm3}, we have the following result.\begin{prop}The optimal consumption is given by $c(t)=(\alpha p(t))^{\frac{1}{R-1}}$, where $p(t)$ satisfies Eq.(\ref{MBD25}).\end{prop}
\section{Application II: Dynamic optimization of pension fund with delayed surplus}
In this section, let us turn to study Example $1.2$ in Section $1$. We will use the results obtained in Section $3$ to derive the optimal control. For simplicity, suppose $g(\cdot)$ is some linear function as follows$$g(y(t)-\kappa y(t-\delta))=\alpha y(t)-\alpha \kappa y(t-\delta)$$where $\alpha, \beta>0$. Then our model can be rewritten as\begin{equation}\label{MBD20}
\left\{\begin{aligned}
dy(t)=&\ \{(r-\alpha)y(t)+\lambda z(t)+\alpha \kappa y(t-\delta)-c(t)\}dt+z(t)dW(t),\,t\in[0,T],\\
y(t)=&\ 0,\, \quad \quad t\in[-\delta,0),\\
y(T)=&\ \xi.
\end{aligned}\right.
\end{equation}Denote the admissible control set by $\mathcal{C}=\{c(\cdot)\in L^2_{\mathbb{F}}(0,T;\mathbb{R}),t\in[0,T]\}.$ It follows that if $\delta$ is sufficiently small, then Eq. (\ref{MBD20}) admits a unique solution pair $(y(\cdot),z(\cdot))$. Introduce the objective functional of the fund manager as follows\begin{equation}J(c(\cdot))=\mathbb{E}\big[\int^T_0Le^{-\rho t}\frac{(c(t))^{1-R}}{1-R}dt\big]-Kx(0),\end{equation}with $L$ and $K$ are positive constants, $\rho$ is a discount factor, and $R\in(0,1)$ is index of risk aversion. The manager aims to maximize the expected objective functional by  taking account both the cumulative consumption and initial reserve requitement. The optimal control problem is to maximize $J(c(\cdot))$ over $\mathcal{C}$. The Hamiltonian function is given by$$H(t, y(t), y(t-\delta), c(t), p(t))=Le^{-\rho t}\frac{(c(t))^{1-R}}{1-R}+\{(r-\alpha)y(t)+\lambda z(t)+\alpha \kappa y(t-\delta)-c(t)\}p(t).$$The adjoint equation is\begin{equation}\label{MBD21}
\left\{\begin{aligned}
dp(t)=&\ \{(\alpha-r)p(t)-\alpha \kappa\mathbb{E}^{\mathcal{F}_t}[p(t+\delta)]\}dt-\lambda p(t)dW(t),\,t\in[0,T],\\
p(0)=&\ K,\\
p(t)=&\ 0,\,\,t\in(T,T+\delta].
\end{aligned}
\right.
\end{equation}Then from Corollary \ref{Cor1}, we have the following result.\begin{prop}If $p(t)$ is the solution of ASDE (\ref{MBD21}), then the optimal consumption is given by $c(t)=\Big(\frac{p(t)e^{\rho t}}{L}\Big)^{-\frac{1}{R}}$ and the optimal fund proportion in risky asset is $\theta(t)=\frac{z(t)}{\sigma y(t)}$ where $(y(t),z(t))$ satisfies (\ref{MBD20}).
\end{prop}In the following, we aim to get the explicit solution of ASDE (\ref{MBD21}). To this end, we first set$$M(t)=e^{\int^t_0-\lambda dW(s)-\frac{1}{2}\int^t_0\lambda^2ds},\,t\in[0,T+\delta].$$It follows that $M(t)$ is an exponential martingale and satisfies$$dM(t)=-\lambda M(t)dW(t).$$Let $p(t)=q(t)M(t)$, where $q(t)$ is a deterministic function defined on $[0,T+\delta]$, then apply It\^{o} formula to $p(t)$, we have\begin{equation}\label{MBD22}dp(t)=q'(t)M(t)dt-\lambda q(t)M(t)dW(t),\,t\in[0,T].
\end{equation}On the other hand, substituting $p(t)=q(t)M(t)$ into Eq.(\ref{MBD21}), we have
\begin{equation}\label{MBD23}
\begin{aligned}dp(t)=&\ \{(\alpha-r)q(t)M(t)-\alpha\kappa q(t+\delta)\mathbb{E}^{\mathcal{F}_t}[M(t+\delta)]\}dt-\lambda q(t)M(t)dW(t)\\
=&\ \{(\alpha-r)q(t)M(t)-\alpha\kappa q(t+\delta)M(t)\}dt-\lambda q(t)M(t)dW(t),\,\,t\in[0,T]
\end{aligned}\end{equation}Comparing (\ref{MBD22}) and (\ref{MBD23}), if the following AODE has a solution\begin{equation}\label{MBD24}
\left\{\begin{aligned}
q'(t)=&\ (\alpha-r)q(t)-\alpha\kappa q(t+\delta),\,\,t\in[0,T],\\
q(0)=&\ K,\\
q(t)=&\ 0,\,\,t\in (T,T+\delta],
\end{aligned}\right.
\end{equation}then $p(t)=q(t)M(t)$ is a solution of ASDE (\ref{MBD21}). The solution of AODE $(\ref{MBD24})$ can be obtained via the characteristic function as follows: $q(t)=Ke^{ht}$ for $t \in [0, T],$ and $q(t)=0$ for $t \in (T, T+\delta].$ Here, $h$ satisfies the following characteristic equation:$$h+\alpha \kappa e^{h\delta}=(\alpha-r).$$Note that $\alpha, r, \kappa$ are the parameter of the state equation, so the above characteristic equation has solution $h$ if the the delayed parameter $\delta$ is small enough. In fact, denote $F(h)=h+\alpha\kappa e^{h\delta},$ then it follows that $\lim_{h\longrightarrow +\infty}F(h)=+\infty.$ In addition,  $F'(h)>0$ so $F(h)$ is an increasing function of $h,$ so there exists unique $h$ such that $F(h)=(\alpha-r),$ thus $q(t)$ as well as $p(t)$ are uniquely determined. One remark to the parameter range. Set $L=\max \{|\alpha-r|, \alpha\kappa,\lambda \},$ then we have the following parameter range to well-poseness of BSDE (\ref{MBD20}) and ASDE (\ref{MBD21}):\begin{equation}\left\{\begin{aligned}
&\ 6L^{2}\delta (1+2\delta^{2}e)<1,\\
&\ 4L^{2}\delta(1+\delta^{2}e)+\delta<1.\\
\end{aligned}\right.
\end{equation}\section{Application III: The dynamic optimization of linear delayed system}
Here, we revisit Example $1.3$ of the backward system with time-delayed generator. The state equation is
given by$$y(t)=\xi-\int_t^{T}\left[\beta_1 y(s)+\beta_2 y(s-\delta)+\gamma_1 z(s)+\gamma_2 z(s-\delta)+\alpha v(s)\right]ds-\int_t^{T}z(s)dW_s$$where $\alpha, \beta_1,\beta_2,\gamma_1,\gamma_2$ are some constants, $v(\cdot)$ is the control process, and the class of admissible controls is denoted by $\mathcal{U}_{ad}=\{v(\cdot)\in L^2_{\mathbb{F}}(0,T;\mathbb{R}),t\in[0,T]\}.$ The dynamic optimization problem is as follows: $$\inf_{v \in \mathcal{U}_{ad}}J(v(\cdot))$$where the objective functional is given by $$J(v(\cdot))=\frac{1}{2}\mathbb{E}\left[\int_0^{T}R(t)v^{2}(t)dt\right]+Ky(0)$$for some constant $K$ and nonnegative function $R(t)$ defined on $[0,T]$. By Corollary \ref{Cor1}, the Hamiltonian function of our optimization problem becomes$$\begin{aligned}
&\ H(t,y(t),y(t-\delta),z(t),z(t-\delta),v(t))\\
=&\ -\frac{1}{2}R(t)v^2(t)+\left(\alpha v(t)+\beta_1y(t)+\beta_2y(t-\delta)+\gamma_1z(t)+\gamma_2z(t-\delta)\right)p(t),
\end{aligned}$$and the adjoint equation becomes\begin{equation}\label{MBDA1}
\left\{\begin{aligned}
dp(t)=&\ \left(-\beta_1p(t)-\beta_2\mathbb{E}^{\mathcal{F}_t}[p(t+\delta)]\right)dt+
\left(-\gamma_1p(t)-\gamma_2\mathbb{E}^{\mathcal{F}_t}[p(t+\delta)]\right)dW(t),\,\,t\in[0,T],\\
p(0)=&\ K,\\
p(t)=&\ 0,\,\,t\in(T,T+\delta].
\end{aligned}\right.\end{equation}Similar to Application II, we can introduce the exponential martingale satisfying: $$dM(t)=\gamma M(t)dW(t)$$where $\gamma$ is some coefficient to be determined, and set $p(t)=q(t)M(t).$ Then we get the following AODE equation system:\begin{equation}
\left\{\begin{aligned}
q'(t)=&\ -\beta_1 q(t)-\beta_2 q(t+\delta),\\
\gamma q(t)=&\ \gamma_1 q(t)+\gamma_2 q(t+\delta).
\end{aligned}\right.\end{equation}The first equation $q(t)$ can be solved using the same method to Eq.(\ref{MBD24}). Based on it, we can plug $q(t)$ into second equation to get the value of $\gamma$ thus the exponential martingale $M(t)$ can be uniquely determined. Consequently, the optimal control is given by $u(t)=\frac{\alpha p(t)}{R(t)}$, where $p(t)=q(t)M(t)$ is the solution of the ASDE (\ref{MBDA1}).


\begin{thebibliography}{99}

\bibitem{alw} A. Augustynowicz, H. Leszczynski and W. Walter (2003). On some nonlinear ordinary differential equations with advanced arguments. \emph{Nonlinear Analysis}, 53, 495-505.

\bibitem{cw} L. Chen and Z. Wu (2010). Maximum principle for the stochastic optimal control problem with delay and
application. \emph{Automatica}, \textbf{46}, 1074-1080.

\bibitem{cwi} K. L. Cooke and J. Wiener (1987). An equation alternately to retarded and advanced type. \emph{Proceeding of the American Mathematical Society}, \textbf{99}, 726-732.

\bibitem{dz} N. Dokuchaev, X. Y. Zhou (1999). Stochastic control with terminal contingent conditions, \emph{J. Math. Anal. Appl.}, \textbf{238}, 143-165.

\bibitem{d1} {\L}. Delong (2011). BSDEs with time-delayed generators of a moving average type with applications to non-monotone preferences. To appear in \emph{Stochastic Models}.

\bibitem{d2} {\L}. Delong (2011). Applications of time-delayed backward stochastic differential equations to pricing, hedging and portfolio management. Working paper.

\bibitem{di} {\L}. Delong and P. Imkeller (2010). Backward stochastic differential equations with time delayed generators-results and counterexamples. \emph{Annals of Applied Probability}, \textbf{20}, 1512-1536.

\bibitem{f} S. Federico (2011). A stochastic control problem with delay arising in a pension fund model. To appear in \emph{Finance and Stochastics}.

\bibitem{hwg} A. J. Hall, G. C. Wake and P. W. Gandar (1991). Steady size distributions for cells in one dimensional plant issues. \emph{J. Math. Bio.}, \textbf{30}, 101-123.

\bibitem{jz} S. Ji and X.Y. Zhou (2006). A maximum principle for stochastic optimal control with terminal state constraints, and its applications.  \emph{Communications in Information and Systems}, \textbf{6}, 321-338.

\bibitem{l} B. Larssen (2002). Dynamic programming in stochastic control of systems with delay. \emph{Stochastics and Stochastics Reports}, \textbf{74}, 651-673.

\bibitem{lz} A. Lim and X.Y. Zhou (2001). Linear-quadratic control of backward stochastic differential equations. \emph{SIAM J. Control Optim.}, \textbf{40}, 450-474.

\bibitem{km} T. Kato and J. B. McLeod (1971). The functional-differential equation. \emph{Bull. Amer. Math. Soc.}, \textbf{77}, 891-937.

\bibitem{kz} M. Kohlmann and X.Y. Zhou (2000). Relationship between backward stochastic differential equations and stochastic controls: a linear-quadratic approach. \emph{SIAM J. Control Optim.}, \textbf{38}, 1392-1407.

\bibitem{my} J. Ma and J. Yong (1999). Forward-backward Stochastic Differential Equations and their
Applications, Lecture Notes in Math. \textbf{1702}, Springer-Verlag.

\bibitem{m1} S. E. A. Mohammed (1984). Stochastic Functional Differential Equations. Pitman Advanced Publishing Program.

\bibitem{m2} S. E. A. Mohammed (1996). Stochastic Differential Equations with Memory:
Theory, Examples and Applications. \emph{Stochastic Analysis and
Related Topics 6.} The Geido Workshop, Progress in Probability, Birkhauser.

\bibitem{or} R. J. Oberg (1969). On the local existence of solutions of certain functional differential equations.
\emph{Proc. Amer. Math. Sco.}, \textbf{20}, 285-302.

\bibitem{os} B. {\O}ksendal and A. Sulem (2001). A maximum principle for optimal
control of stochastic systems with delay, with applications to
finance. \emph{Optimal Control and Partial Differential Equations}, eds J. L. Menaldi, E. Rofman and A. Sulem, IOS Press, Amsterdam, 64-79.

\bibitem{po} G. P. Papavassilopoulos and G. J. Olsder (1984). On a linear differential equation of the advanced type.
\emph{J. Math. Anal. Appl.}, 103, 74-82.

\bibitem{kpq} N. El Karoui, S. Peng and M. C. Quenez (2001). A dynamic maximum principle for the optimization of recursive utilities under constrains. \emph{Annals of Applied Probability}, \textbf{11}, 664-693.

\bibitem{py} S. Peng and Z. Yang (2009). Anticipated backward stochastic differential
equation, \emph{Annals of Probability}, \textbf{37}, 877-902.

\bibitem{R} R. T. Rockafellar (1970). Convex Analysis. Princeton, NJ: Princeton Univ. Press.

\bibitem{y} T. Yoneda (2006). On the functional-differential equation of advanced type.  \emph{J. Math. Anal. Appl.}, \textbf{332}, 487-496.

\bibitem{yz} J. Yong and X. Y. Zhou (1999). Stochastic controls: Hamiltonian Systems and
HJB equations. Springer-Verlag, New York.
\end{thebibliography}
\end{document}